\documentclass[12pt]{amsart}
\usepackage{yrodro}
\usepackage{epstopdf}

\newcommand  {\DEES} {\text{\sc downs}}
\newcommand  {\AREA} {\text{\sc area}}
\newcommand  {\CORNERS} {\text{\sc corners}}
\newcommand  {\CINDEX} {\text{\sc c-index}}
\newcommand  {\qbinom} [2]{\genfrac[]{0pt}{}{#1}{#2}_q}
\newcommand  {\qtbinom} [2]{\genfrac[]{0pt}{}{#1}{#2}_{q,t}}
\renewcommand{\l} {q}
\newcommand  {\m} {t}    
\newcommand  {\D} {\text{\sf D}}
\renewcommand{\O} {\mathcal{O}}
\newcommand  {\G} [2]{G_{#1,#2}}
\renewcommand{\P} [2]{\mathcal{P}_{#1,#2}}
\renewcommand{\R} {\text{\sf R}}
\newcommand{\wa}{\pmb{a}}
\newcommand{\wb}{\pmb{b}}
\newcommand{\wu}{\pmb{u}}
\newcommand{\wv}{\pmb{v}}
\newcommand{\ww}{\pmb{w}}

\newcommand{\xmathpalette}[2]{\mathchoice
  {#1\displaystyle\textfont{#2}}%
  {#1\textstyle\textfont{#2}}
  {#1\scriptstyle\scriptfont{#2}}%
  {#1\scriptscriptstyle\scriptscriptfont{#2}}%
}
\makeatletter
\newcommand{\Ro}{\mathsf{R}\mkern-3.7mu{\xmathpalette\R@o\relax}}
\newcommand{\oD}{\mathsf{\xmathpalette\o@D\relax\mkern-3.8mu D}}

\newcommand{\R@o}[3]{%
  \scalebox{0.6}{%
    \raisebox{\dimexpr\height-2\fontdimen22#22}{%
      $\m@th#1\bullet$%
    }%
  }%
}
\newcommand{\o@D}[3]{%
  \raisebox{\fontcharht#2\fam`D}{%
    \scalebox{0.65}{%
      \raisebox{-\height}{$\m@th#1\bullet$}%
    }%
  }%
}
\makeatother

\newcommand\vsp[1][.3em]{%
  {\ }\mbox{\kern.06em\vrule height.3ex}%
  \vbox{\hrule width#1}%
  \hbox{\vrule height.3ex}{\ }}

\theoremstyle{plain}

\newtheorem{Thm}{Theorem}
\newtheorem{Prop}[Thm]{Proposition}

\newtheorem{Cor}[Thm]{Corollary}
\newtheorem{fact}{Fact}

\newtheorem{innercustomgeneric}{\customgenericname}
\providecommand{\customgenericname}{}
\newcommand{\newcustomtheorem}[2]{%
  \newenvironment{#1}[1]
  {%
   \renewcommand\customgenericname{#2}%
   \renewcommand\theinnercustomgeneric{##1}%
   \innercustomgeneric
  }
  {\endinnercustomgeneric}
}
\newcustomtheorem{customlemma}{Lemma}
\newcounter{tmp}

\theoremstyle{remark}
\newtheorem{Rem}[Thm]{Remark}

\begin{document}

\title{Scrambled Vandermonde Convolutions of Gaussian Polynomials}
\author{Magnus Aspenberg}
\address{Centre for Mathematical Sciences, Lund University}
\email{magnusa@maths.lth.se}
\author{Rodrigo A. P{\'e}rez}
\address{Department of Mathematical Sciences, IUPUI}
\email{rperez@math.iupui.edu}

\maketitle


\begin{abstract}
  It is well known that Gaussian polynomials (i.e., $q$-binomials) describe the distribution of the $\AREA$ statistic on monotone paths in a rectangular grid. We introduce two new statistics, $\CORNERS$ and $\CINDEX$; attach ``ornaments'' to the grid; and re-evaluate these statistics, in order to argue that all scrambled versions of the $\CINDEX$ statistic are equidistributed with $\AREA$. Our main result is a representation of the generating function for the bi-statistic $(\CINDEX,\CORNERS)$ as a two-variable Vandermonde convolution of the original Gaussian polynomial. The proof relies on explicit bijections between differently ornated paths.
\end{abstract}

Let $G_{m,n}$ be a rectangular grid of $m$ vertical by $n$ horizontal unit squares, and consider the collection $\P{m}{n}$ of paths of length $m+n$ joining the upper-left corner to the lower-right (thus, moving only by down/right steps).

The cardinality of $\P{m}{n}$ is $\binom{m+n}{m} = \binom{m+n}{n}$. A classic generalization of this elementary fact states that the $q$-binomial $\qbinom{m+n}{m}$ (aka Gaussian polynomial) describes a gradation of $\P{m}{n}$, so that the coefficient of $\l^a$ counts the number of paths that cover an area of $a$ squares (cf. Figure~\ref{fig:qt-2x4}). \\

Our aim here is to pursue further this combinatorial structure of $\P{m}{n}$. As groundwork, we describe in \S\ref{sect:Defns} a symmetric version of the Vandermonde formula; give basic facts about the auxiliary statistics $\DEES$ and $\AREA$; and begin to study the statistics $\CORNERS$ and $\CINDEX$ that constitute our main subject. In particular, we illustrate our later methods by sketching a proof that $\CINDEX$ is equidistributed with $\AREA$. In \S\ref{sect:Ornaments} we introduce certain ornaments on $G_{m,n}$ that alter the $\CORNERS$ and $\CINDEX$ values on any given path. It turns out that this ``scrambled'' version of $\CINDEX$ is {\em still equidistributed} with $\AREA$, up to a constant shift (i.e., is generated by $\qbinom{m+n}{m} \cdot \l^s$ for some concrete $s$).

The equidistribution claims of \S\ref{sect:Defns} and \S\ref{sect:Ornaments} are initially discussed only in outline, because they are straightforward consequences of our main result, Theorem~\ref{thm:Main}: For any given choice $\O$ of ornaments on $G_{m,n}$, the generating function of the scrambled bi-statistic $(\CINDEX,\CORNERS)$ is a two-variable polynomial $\qtbinom{m+n}{m}^{\O}$, which is described by Formula~\eqref{eqn:Main} as a shifted Vandermonde convolution that depends on $\O$. The proof in \S\ref{sect:Proof} uses {\em explicit bijective maps} between differently ornated copies of $\G{m}{n}$.

In \S~\ref{sect:Final} we infer some consequences of~\eqref{eqn:Main}, including proper proofs of the equidistribution claims, and discuss further results and open problems.

\section{Definitions}
\label{sect:Defns}

\subsection{$\boldsymbol{\l}$-Vandermonde convolutions}
\label{sect:Vandermonde}
In its standard presentation, the $\l$-Vandermonde formula reads
\begin{equation}
 \label{eqn:qVandermonde}
  \qbinom{x+y}{m} =
  \sum_j \qbinom{x}{m-j} \cdot \qbinom{y}{j} \cdot \l^{(x-m+j)j}.
\end{equation}
We make sense of this as follows. Fix $x,y$, and let $m$ vary; then,~\eqref{eqn:qVandermonde} decomposes {\em every} $\l$-binomial of level $x+y$ as a convolution of $x$-level and $y$-level binomials. In this work, however, we require a {\em different} interpretation.

\smallskip
Replace $\qbinom{x}{m-j}$ with the equivalent $\qbinom{x}{x-(m-j)}$, and consider the variable changes
\[x=m+(r-d) \quad,\quad y=n+(d-r) \quad,\quad j=c-r.\]
Then,~\eqref{eqn:qVandermonde} becomes
\begin{equation}
 \label{eqn:Our-qVandermonde}
  \qbinom{m+n}{m} =
  \sum_c \qbinom{m+r-d}{c-d} \cdot \qbinom{n+d-r}{c-r} \cdot \l^{(c-d)(c-r)}
\end{equation}
\[
\left( \begin{tabular}{c}
{\footnotesize Here, as elsewhere, bounds on the summation variable are left implicit,} \\
{\footnotesize since binomials out of range are always conveniently equal to zero}
\end{tabular}\right)
\]

\medskip
We read this, more symmetric version of the $\l$-Vandermonde formula, as providing a family of convolution representations for one and the same binomial. Whenever the difference $d-r$ equals a given value $Q$, equation~\eqref{eqn:Our-qVandermonde} decomposes $\qbinom{m+n}{m}$ as a convolution of $(m-Q)$-level and $(n+Q)$-level binomials. Notice that such representation depends only on the difference $Q$ and not on the specific choice of $d$ and $r$ (see comment below~\eqref{eqn:Our-qVandermonde}).

\subsection{Definitions}
\label{sect:Definitions}
The grid $G_{m,n}$ is delineated by $m+1$ horizontal lines, numbered 0 through $m$, and $n+1$ vertical lines, numbered 0 through $n$. The paths in $\P{m}{n}$ advance by down/right steps that Start at the upper-left corner $S$, and Finish at the lower-right corner $F$. The coordinates are set so that $S$ is at $(0,0)$, and $F$ at $(m,n)$, as in matrix notation.

To each path there is naturally associated a word in the alphabet $\{\D,\R\}$, with $m$ copies of $\D$ and $n$ copies of $\R$; whence $|\P{m}{n}| = \binom{m+n}{m} = \binom{m+n}{n}$. We obviate the distinction between path and word, as there is no risk of confusion; in this spirit, consider a word $\ww \in \P{m}{n}$. The action of turning a block $\R\D$ of $\ww$ into $\D\R$ or the other way around is a {\em swap}. Since the numbers of $\D$s and $\R$s are unchanged, swaps result in new elements of $\P{m}{n}$.

\subsection{Statistics}
\label{sect:Statistics}
Here we describe four statistics on $\P{m}{n}$, and state a few auxiliary facts that are well known, and/or easy to prove. At the end, we give a new characterization of $\l$-binomial coefficients in Fact~4. Although this will follow directly from our main result, we sketch an alternative proof that illustrates the connection to symmetric Vandermonde convolutions.

\renewcommand{\theenumi}{\alph{enumi})}
\renewcommand{\labelenumi}{\theenumi}.
\begin{enumerate}
  \item The statistic $\DEES$ is the sum of the positions within a word $\ww \in \P{m}{n}$,
    that hold a $\D$. Since the minimum of $\DEES$ is attained only by $m\D n\R$, and any other word allows a swap that decreases $\DEES$ by $1$, we find

    \begin{fact}
     \label{fact:Pconnected}
       $\P{m}{n}$ is connected under swaps.
    \end{fact}
  \item The statistic $\AREA$ counts the number of unit squares below a given path. The
    following is well known (see, e.g., \cite{Stanley-EnumerativeCombinatorics-1}, p.~29):

  \begin{fact}
   \label{fact:BinomialArea}
    The generating function of $\AREA$ over all paths in $\P{m}{n}$ is the $\l$-binomial $\qbinom{m+n}{m}$. That is, the coefficient of $\l^a$ in $\qbinom{m+n}{m}$ equals the number of paths that cover exactly $a$ squares.
  \end{fact}

  Every swap changes both $\DEES$ and $\AREA$ by the same amount $-1$ or $+1$. Since $\AREA(m\D n\R) = 0$, and $\DEES(m\D n\R) = \tfrac{m(m+1)}{2}$, we find

  \begin{fact}
   \label{fact:AREAandDEESequidistributed}
    For every path $p \in \P{m}{n}$, $\DEES(p) = \AREA(p) + m(m+1)/2$. In particular, $\DEES$ and $\AREA$ are equidistributed up to a shift by $\l^{m(m+1)/2}$.
  \end{fact}
  \item The statistic $\CORNERS$ counts the number of occurrences of the pair $\R\D$ in a
    path; i.e., the number of grid nodes where a horizontal step is followed by a vertical step. We reserve the denomination of {\em corner} to such nodes; in particular, ``corners'' of the type $\D\R$ will {\em not} be considered (but see \S~\ref{sect:Final}).
  \item The {\em c-index} of a corner is its distance (in the grid) from the starting node
    $S$, or equivalently, the sum of the corner's coordinates. The statistic $\CINDEX$ is the sum of the c-indices of all corners in a path. In terms of words, we mark the places where $\R$ is followed by $\D$, and sum the positions of the marks:
  \[\label{pg:sampleWord}
    p =
    \R \underset{1}{\phantom{\bullet}} \R \underset{2}{\bullet} \D
      \underset{3}{\phantom{\bullet}} \R \underset{4}{\phantom{\bullet}} \R \underset{5}{\bullet} \D \quad,\qquad \CINDEX(p) =
    2+5 = 7
  \]
  It is important to highlight how, in contrast to counting letter positions (as when evaluating $\DEES$), here we enumerate the {\em spaces between} letters. We call these spaces {\em gaps}.
\end{enumerate}

Figure~\ref{fig:qt-2x4} depicts the $\binom{2+4}{2} = 15$ paths in $\P{2}{4}$, together with their corresponding words, and values for $\DEES$, $\AREA$, $\CORNERS$, and $\CINDEX$. In this concrete example it is easy to verify that the coefficients of $\qbinom{2+4}{2} = 1 + \l + 2\l^2 + 2\l^3 + 3\l^4 + 2\l^5 + 2\l^6 + \l^7 + \l^8$ describe both the $\AREA$ and $\CINDEX$ statistics. \underline{This is a general truth}:
\begin{fact}
 \label{fact:AREAandCINDEXequidistributed}
  The generating function of $\CINDEX$ over all paths in $\P{m}{n}$ is $\qbinom{m+n}{m}$; i.e., $\CINDEX$ is equidistributed with $\AREA$.
\end{fact}
\begin{figure}[h]
  \includegraphics[width=5.5in]{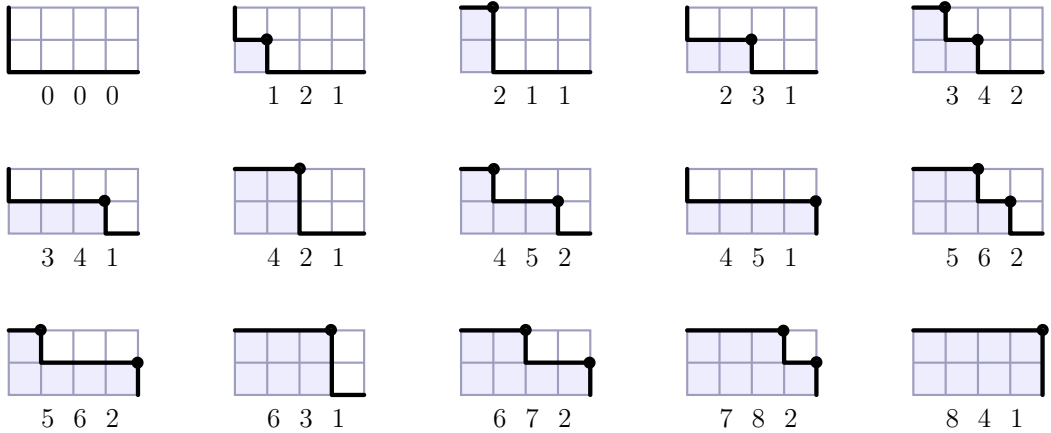}
  \caption{Below each of the 15 paths in $\P{2}{4}$ are the statistics $\AREA,\,\CINDEX$, and $\CORNERS$. The distribution of {\em both} $\AREA$ and $\CINDEX$ is described by $1+\l+2\l^2+2\l^3+3\l^4+2\l^5+2\l^6+\l^7+\l^8$, even though their values match {\em only} on the first path. As noted after equation~\eqref{eqn:qt-2x4}, the distribution of $\CORNERS$ is $1,8,6$. \hfill {\footnotesize($\DEES$ not shown)} }
 \label{fig:qt-2x4}
\end{figure}

\medskip\noindent
{\bf Proof sketch of Fact~\ref{fact:AREAandCINDEXequidistributed}.}
Only those nodes to the right of $S$ and above $F$ can hold a corner; i.e., corners are confined to the sub-grid delineated by horizontal lines $0$ through $m-1$, and vertical lines $1$ through $n$. A path with $c$ corners is uniquely determined by a choice of $c$ horizontal, and $c$ vertical lines in this sub-grid (consecutive horizontal/vertical pairs will form consecutive corner coordinates, which then determine the path). Since the c-index of a corner is precisely the sum of these horizontal and vertical coordinates, Fact~\ref{fact:AREAandDEESequidistributed} gives by induction that the generating function of $\CINDEX$, {\em restricted} to paths with $c$ corners, is $\big( \qbinom{m}{c}\,\l^{c(c+1)/2} \big) \cdot \big(\qbinom{n}{c}\, \l^{(c-1)c/2} \big)$. Notice that the powers of $\l$ add up to $c^2$, so the generating function for $\CINDEX$ over {\em all} paths in $\P{m}{n}$ is
\begin{equation}
 \label{eqn:1stVandermonde}
  \sum_{p \in \P{m}{n}} \l^{\CINDEX(p)} =
  \sum_{c} \qbinom{m}{c} \qbinom{n}{c}\, \l^{c^2},
\end{equation}
which, according to~\eqref{eqn:Our-qVandermonde} (with $d=r=0$), equals $\qbinom{m+n}{m}$ as claimed. \qed

\subsection{A polynomial in two variables}
\label{sect:qtBinomials}
Not only did the argument above establish that $\CINDEX$ is equidistributed with $\AREA$, but it actually unearthed a finer gradation of $\qbinom{m+n}{m}$ that {\em distinguishes paths} with different values of $\CORNERS$\footnote{An analogous gradation exists for $\AREA$. In that case, $c$ represents the side of the largest square sitting at $(m,0)$, that fits under a given path $p$. Then the analog of~\eqref{eqn:1stVandermonde} is obtained by induction, using the square's inner corner to split $p$ into two shorter sub-paths. Details are left to the reader.}. Thus, we may include the factor $\m^c$ in each summand of~\eqref{eqn:1stVandermonde}, to obtain the 2-variable polynomial
\begin{equation}
 \label{eqn:1st-qtVandermonde}
  \qtbinom{m+n}{m} =
  \sum_{c} \qbinom{m}{c} \qbinom{n}{c}\, \l^{c^2}\m^c,
\end{equation}
where the coefficient of $\l^a \m^b$ counts paths $p$ with $\CINDEX(p) = a$, and $\CORNERS(p) = b$. 

For instance (ref. Figure~\ref{fig:qt-2x4Scrambled}), we see from Figure~\ref{fig:qt-2x4} that $\qtbinom{2+4}{2}$ equals $\big( 1 + \l\m + 2\l^2\m + 2\l^3\m + 2\l^4\m + \l^4\m^2 + \l^5\m + \l^5\m^2 + 2\l^6\m^2 + \l^7\m^2 + \l^8\m^2\big)$; a result that we will most often write in the form
\begin{equation}
 \label{eqn:qt-2x4}
  \qtbinom{2+4}{2} =
  1 \,+\,
    \big( \l + 2\l^2 + 2\l^3 + 2\l^4 + \l^5 \big)\,\m \,+\,
    \big( \l^4 + \l^5 + 2\l^6 + \m^7 + \m^8 \big)\,\m^2,
\end{equation}
because we are interested in recovering the $\CINDEX$ count when paths are restricted to a given number of corners. Notice that $\m=1$ recovers $\qbinom{2+4}{2}$, while $\l=1$ yields $1+8\m+6\m^3$, the generating function of $\CORNERS$ on $\P{2}{4}$.

\section{Scramblers}
\label{sect:Ornaments}
Let $\O = (H,V)$  be a pair of sets $H = \{h_1,\ldots, h_d \} \subset \{0,\ldots,m-1\}$, and $V =\{v_1,\ldots, v_r \} \subset \{1,\ldots,n\}$. For later use, we let $s$ stand for the {\em ornament sum} $\sum h_i + \sum v_j$, and $Q$ for the difference $d-r$ (as in \S\ref{sect:Vandermonde}).

For every $h_i \in H$, place an ornament to the right of node $(h_i,n) \in \G{m}{n}$. For every $v_j \in V$ place an ornament above node $(0,v_j)$. We visualize the ornaments as laser guns shining along their respective horizontal/vertical lines (cf.~\cite{Lyness-1941}). Each ray hits a given path $p$ at a specific node, which we promote to a ``virtual'' corner. We denote by $W^{\O}$ any subset $W \subset \P{m}{n}$ where all paths have virtual corners additionally marked. Since the values of both $\CORNERS$ and $\CINDEX$ are altered, we call $\O$ a {\em scrambler} on $\G{m}{n}$. Our goal is to to describe the generating function $\qtbinom{m+n}{m}^{\O}$ of the scrambled bi-statistic $(\CINDEX,\CORNERS)$.

Let us upgrade the notation for path words. An ornament on the vertical line $v_j$ guarantees a corner (true or virtual) {\em after} the $v_j$-th $\R$-step. We highlight this fact by styling the $v_j$-th copy of $\R$ thus: $\Ro$. Similarly, an ornament on the horizontal line $h_i$ guarantees a corner {\em before} the $h_i$-th $\D$-step, so the $h_i$-th copy of $\D$ is styled $\oD$.

\begin{figure}[h]
  \includegraphics[width=5.5in]{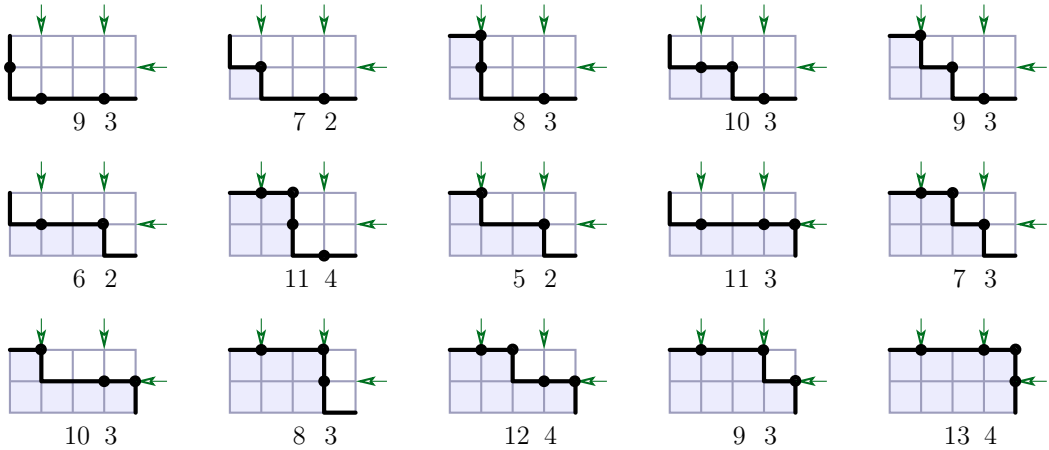}
  \caption{The 15 paths in $\P{2}{4}^{( \{1\},\,\{1,3\} )}$. Since $\AREA$ is as in Figure~\ref{fig:qt-2x4}, only $\CINDEX$ and $\CORNERS$ are displayed. As per the comment after equation~\eqref{eqn:qt-2x4Scrambled}, the  distribution of $\CORNERS$ is now $3,9,3$, in contrast to the $1,8,6$ of Figure~\ref{fig:qt-2x4}.}
 \label{fig:qt-2x4Scrambled}
\end{figure}

As illustration, Figure~\ref{fig:qt-2x4Scrambled} redoes the enumeration of Figure~\ref{fig:qt-2x4}, modified by the scrambler $\O = \big( \{1\},\,\{1,3\} \big)$. The sample path of page~\pageref{pg:sampleWord} now has two new virtual corners at gaps 1 and 4 (the corner at gap 5 is true, even though it is marked by a horizontal ornament):
\[
  p =
  \Ro \underset{1}{\phantom{\bullet}} \R \underset{2}{\bullet} \D
    \underset{3}{\phantom{\bullet}} \Ro \underset{4}{\phantom{\bullet}} \R \underset{5}{\bullet} \oD \qquad \CINDEX(p) =
  1+2+4+5 = 12
\]

The resulting generating function on $\P{2}{4}^{(\{1\},\,\{1,3\})}$ is

\begin{equation}
 \label{eqn:qt-2x4Scrambled}
  \qtbinom{2+4}{2}^{(\{1\},\,\{1,3\})} =
  (\l^5+\l^6+\l^7) \m^2 +
  (\l^7+2\l^8+3\l^9+2\l^{10}+\l^{11}) \m^3 +
  (\l^{11}+\l^{12}+\l^{13}) \m^4.
\end{equation}

Just as we noted apropos of~\eqref{eqn:qt-2x4}, we can substitute $\l=1$ to reveal 3 paths with two corners, 9 paths with three corners, and 3 paths with four corners: the distribution of $\CORNERS$ has changed. However, the distribution of $\CINDEX$ still coincides with $\qbinom{2+4}{2}$, only shifted by a constant factor of $\l^5$.

\newpage

After this illustrative example, we are ready to state our main result (recall that the ornament sum $s$ equals $\sum h_i + \sum v_j$).

\begin{Thm}
 \label{thm:Main}
  For any scrambler $\O = (H,V)$ on $G_{m,n}$, with $|H| = d$ and $|V| = r$, the generating function of the bi-statistic $(\CINDEX,\CORNERS)$ over all paths in $\P{m}{n}^{\O}$ is the polynomial
  \begin{equation}
   \label{eqn:Main}
    \qtbinom{m+n}{m}^{\O} =
    \l^s \cdot \sum_c \left\{
      \qbinom{m+r-d}{c-d} \cdot \qbinom{n+d-r}{c-r} \cdot \l^{(c-d)(c-r)} \cdot \m^c
    \right\}.
  \end{equation}
\end{Thm}
\smallskip\begin{center}{\small $\big( \text{recall }s=\sum h_i+\sum v_j \big)$}\end{center}

\medskip\noindent{\bf Observations.}
Of course, \eqref{eqn:1st-qtVandermonde} is the special case of this equation with an empty scrambler. Notice as well that substituting $\m=1$ reduces $\qtbinom{m+n}{m}^{\O}$ to the symmetric version~\eqref{eqn:Our-qVandermonde} of the $q$-Vandermonde formula, except for the extra factor $\l^s$. In other words,

\begin{Cor}
 \label{corol:FullEquidistributionClaim}
  The statistic $\CINDEX$ over $\P{m}{n}^{\O}$ is equidistributed with $\CINDEX$ over $\P{m}{n}$ (and thus, with $\AREA$ as well), up to a shift by the factor $\l^s$.
\end{Cor}

\section{Proof of the Theorem}
 \label{sect:Proof}

\subsection{Strategy of proof}
\label{sect:ProofStrategy}
We will proceed by restricting attention to the set of scramblers with a fixed value of the difference $d-r$ (which we called $Q$ in \S\ref{sect:Defns}). First, given $\O$ and $\O'$ with $d-r = d'-r'$, we use lemmas~\ref{lemma:OrnamentPairCancellation} and~\ref{lemma:OrnamentShift-H},~\ref{lemma:OrnamentShift-V} to construct an explicit bijection between $\P{m}{n}^{\O}$ and $\P{m}{n}^{\O'}$, so that corresponding paths have $\CINDEX$ differing by the constant value $s-s'$. This implies equidistribution of $(\CINDEX,\CORNERS)$ between $\P{m}{n}^{\O}$ and $\P{m}{n}^{\O'}$, up to the correct shift in $\l$ and $\m$. Afterward, we consider the scrambler with given $Q$, which has smallest ornament sum $s$. It has only one kind of ornaments, either horizontal or vertical, all at their lowest positions. In Proposition~\ref{prop:ReducedVandermonde} we will prove formula~\eqref{eqn:Main} in just such a situation, thus establishing the overall result.

\subsection{A pair of unscrambling bijections}
In the first auxiliary lemma below, notice that the differences $d-r$ and $d'-r'$ are equal, while the ornament sums satisfy $s'=s-1$.
\begin{lemma}
 \label{lemma:OrnamentPairCancellation}
   Let $\O = \{H,V\}$ be a scrambler such that $H$ and $V$ include the ornaments $0$ and $1$ respectively. Let $\O'$ be the modified scrambler where those two ornaments are removed. Then there is a bijection from $\P{m}{n}^{\O}$ to $\P{m}{n}^{\O'}$ that decreases both $\CORNERS$ and $\CINDEX$ by $1$.
\end{lemma}

\begin{Cor}
 \label{corol:PairCancellationImpliesMain}
  Formula~\eqref{eqn:Main} is satisfied by $\O$ if and only if it is satisfied by $\O'$.
\end{Cor}
\begin{proof}
  Let $\pi \in \P{m}{n}^{\O}$ and $\pi' \in \P{m}{n}^{\O'}$ be a bijective pair of paths. Lemma~\ref{lemma:OrnamentPairCancellation} implies that their respective monomials satisfy
  \[
    \l^{\CINDEX(\pi')}\cdot\m^{\CORNERS(\pi')} =
    (\l\m)^{-1} \cdot \l^{\CINDEX(\pi)}\cdot\m^{\CORNERS(\pi)},
  \]
  and consequently (changing summation variable to $c'=c-1$, and noting that removal of one horizontal and one vertical ornament implies $(d',r') = (d-1,r-1)$),
  \begin{multline*}
    \qtbinom{m+n}{m}^{\O'} = (\l\m)^{-1} \cdot \qtbinom{m+n}{m}^{\O} = \\
    \l^{s-1} \cdot \sum_c \left\{
        \qbinom{m+r-d}{c-d} \cdot \qbinom{n+d-r}{c-r} \cdot \l^{(c-d)(c-r)} \cdot \m^{c-1}
      \right\} = \\
    \l^{s'} \cdot \sum_{c'} \left\{
        \qbinom{m+r'-d'}{c'-d'} \cdot \qbinom{n+d'-r'}{c'-r'} \cdot \l^{(c'-d')(c'-r')} \cdot \m^{c'}
      \right\}
  \end{multline*}
\end{proof}

\begin{proof}[Proof of Lemma~\ref{lemma:OrnamentPairCancellation}]
  We will split $\P{m}{n}$ into classes, and define a bijection within each class. The bijection on $\P{m}{n}$ is trivially the union of class-by-class bijections. \\

  Since $H,V \neq \varnothing$, and $0 \in H,\, 1\in V$, every word in $\P{m}{n}$ (before adding ornaments) is of the form {\large $\wa \R\wb\, \wv$}, where
  \begin{itemize}
    \item[--] Any of the words $\wa,\wb,\wv$ can be empty, but there is at least one $\D$ in $\wa\wb\wv$.
    \item[--] The concatenation $\wa\wb$ consists exclusively of $\D$s.
    \item[--] If $\wv$ is not empty, it begins with $\R$.
  \end{itemize}
  Any choice of $\wv$ determines a class $S_{\wv}$ of words in $\P{m}{n}$ which differ {\em only} in the position of the first $\R$, ``floating in a sea'' of $\D$s, so that $S_{\wv}$ consists of $|\wa\wb|+1$ words. If $\wa,\wb$ are both empty, the class $S_{\wv}$ consists of a single word in which the first $\R$ is followed by a second $\R$, and the first $\D$ is preceded by (some) $\R$. In this case the trivial bijection from $S_{\wv}^{\O}$ to $S_{\wv}^{\O'}$ satisfies the conclusion, because removal of the first $\R$-ornament reduces both $\CORNERS$ and $\CINDEX$ counts by 1, while removal of the first $\D$-ornament does not change either statistic.

  Otherwise, $\wa\wb$ contains the first $\D$, so the ornated version of $\wv$ is identical for all words in ${S_{\wv}}^{\O} \cup {S_{\wv}}^{\O'}$, and thus, we only need to track the contributions to $\CORNERS$ and $\CINDEX$ that occur before $\wv$. To simplify the argument, let us ignore WLOG any contribution to $\CORNERS$ or $\CINDEX$ originating in $\wv$.

  Consider four types of words, both in ${S_{\wv}}^{\O}$ and in ${S_{\wv}}^{\O'}$ (keeping in mind that the floating letter is $\Ro$ for $\O$, and $\R$ for $\O'$):
  {
  \renewcommand{\theenumi}{\arabic{enumi}}
  \renewcommand{\labelenumi}{(\theenumi)}
  \begin{enumerate}
    \item\label{item:type1} $\wa,\wb \neq \varnothing$; the floating letter faces a $\oD$.
    \item\label{item:type2} $\wa,\wb \neq \varnothing$; the floating letter faces a $\D$.
    \item\label{item:type3} $\wa = \varnothing$; i.e., the floating letter is leftmost.
    \item\label{item:type4} $\wb = \varnothing$; i.e., the floating letter faces $\wv$.
  \end{enumerate} }

  Since the floating letter in types~\ref{item:type1}, \ref{item:type2} forms a true corner regardless of ornaments, it makes no difference to $\CORNERS$ and $\CINDEX$ whether it is $\Ro$ or $\R$. Moreover, the only other change from $\O$ to $\O'$ is the removal of the ornament in the first $\D$ of $\wa$, whose c-index is $0$. Thus,

  \begin{itemize}
    \item Each word $\ww \in {S_{\wv}}^{\O}$ of type~\ref{item:type1}~or~\ref{item:type2}, has the same c-index as its counterpart $\ww' \in {S_{\wv}}^{\O'}$.
    \item All words of type~\ref{item:type1} have exactly $\big( \text{number of }\oD\text{s} \big)$ corners. This statistic is larger by 1 in ${S_{\wv}}^{\O}$ than in ${S_{\wv}}^{\O'}$.
    \item All words of type~\ref{item:type2} have  exactly $\big( \text{number of }\oD\text{s}+1 \big)$ corners. As above, this statistic is larger by 1 in ${S_{\wv}}^{\O}$.
  \end{itemize}

  Now let us impose linear orders on words of type~\ref{item:type1} and type~\ref{item:type2}, in each of ${S_{\wv}}^{\O}$ and ${S_{\wv}}^{\O'}$ independently. We declare a word of type~\ref{item:type1} to be lower in the order if its floating letter is located more to the right. Visually, in an ordered list of words of type~\ref{item:type1}, the floating letter ``floats'' from right to left. Thus,

  \begin{itemize}
    \item At each step in either ordered list, $\CINDEX$ increases by $1$ (because one $\oD$ shifts one space to the right).
  \end{itemize}

  In a similar fashion, we order words of type~\ref{item:type2} so that the floating letter floats from left to right. We find,
  \begin{itemize}
    \item At each step in either ordered list, $\CINDEX$ increases by $1$ (because a $\D\overbrace{\oD\ldots\oD}^k$ block shifts to the left, increasing $\CINDEX$ by $(k+1)-k$).
  \end{itemize}

  Next we string these orders together (again, in each of ${S_{\wv}}^{\O}$ and ${S_{\wv}}^{\O'}$ independently), so that type~\ref{item:type1} precedes the single word of type~\ref{item:type3}, which precedes type~\ref{item:type2}.

  \begin{itemize}
    \item In $S_{\wv}^{\O}$, the word $\Ro\, \wa\wb\, \wv$ of type~\ref{item:type3} continues the floating pattern in type~\ref{item:type1}, so $\CORNERS = \#\oD$, and $\CINDEX$ {\em ends} the increase-by-1 rule from type~\ref{item:type1}.
    \item In $S_{\wv}^{\O}$, the word $\R\, \wa\wb\, \wv$ of type~\ref{item:type3} initiates the floating pattern in type~\ref{item:type2}, so $\CORNERS = \#\oD+1$, and $\CINDEX$ {\em starts} the increase-by-1 rule in type~\ref{item:type2}.
  \end{itemize}

  Finally, add the single word of type~\ref{item:type4} to the ordered lists; but whereas we place $\wa\wb\,\Ro\,\wv$ last in the ${S_{\wv}}^{\O}$ order, $\wa\wb\,\R\,\wv$ will be listed first in the ${S_{\wv}}^{\O'}$ order. Then

  \begin{itemize}
    \item The word $\wa\wb\,\Ro\,\wv \in S_{\wv}^{\O}$ has the same number of corners as type~\ref{item:type2}, while $\wa\wb\,\R\,\wv \in S_{\wv}^{\O'}$ has the same number of corners as type~\ref{item:type1}.
    \item The word $\wa\wb\,\Ro\,\wv$ {\em ends} the $\CINDEX$ increase-by-1 rule of type~\ref{item:type2}, while $\wa\wb\,\R\,\wv$ {\em starts} the increase-by-1 rule of type~\ref{item:type1} (and for the same reasons as above).
  \end{itemize}

  \begin{table}[h]\begin{center}
  \begin{tabular}{|c@{\qquad$\longleftrightarrow$\qquad}c|}
    \hline
    $\oD\, \boxed{\Ro}\, \oD\, \D \,\,\, \wv$ &
    $\D\, \oD\, \D\, \boxed{\R} \,\,\, \wv$ \\
    $\boxed{\Ro}\, \oD\, \oD\, \D \,\,\, \wv$ &
    $\D\, \boxed{\R}\, \oD\, \D \,\,\, \wv$ \\
    $\oD\, \oD\, \boxed{\Ro}\, \D \,\,\, \wv$ &
    $\boxed{\R}\, \D\, \oD\, \D \,\,\, \wv$ \\
    $\oD\, \oD\, \D\, \boxed{\Ro} \,\,\, \wv$ &
    $\D\, \oD\, \boxed{\R}\, \D \,\,\, \wv$ \\
    \hline
  \end{tabular}\end{center}
  \caption{A floating $\protect\Ro$ within $\protect\oD\protect\oD\D$ turns into a floating $\R$ floating within $\D\protect\oD\D$\,.}
   \label{table:Floating-1}
  \end{table}
  
  We have defined linear orders in both ${S_{\wv}}^{\O}$ and ${S_{\wv}}^{\O'}$, the only difference among which is the position of the type~\ref{item:type4} word. Let us list the words in ${S_{\wv}}^{\O}$ as $\{ \ww_0,\ldots\,\ww_{|ab|} \}$, and those in ${S_{\wv}}^{\O'}$ as $\{ \ww'_0,\ldots\,\ww'_{|ab|} \}$, with indices following the respective orders. Together, the bulleted facts above imply that in {\em both lists}, $\CORNERS$ is piecewise constant, jumping up by 1 at the type~\ref{item:type3} word, while $\CINDEX$ increases by 1 at every step except at type~\ref{item:type3}, where the increase is common to both lists. It follows that for every bijective pair $\ww_i \leftrightarrow \ww_i'$, the offset in either statistic is the same regardless of the index $i$.

  To finish the proof, notice that the last two bullets imply $\CORNERS(\ww_i') = \CORNERS(\ww_i)-1$, and $\CINDEX(\ww_i') = \CINDEX(\ww_i)-1$ for both $i=0$ and $|\wa\wb|$, so the offset is $-1$ for both $\CORNERS$ and $\CINDEX$, as claimed.
\end{proof}

\bigskip
We construct now a bijection between two scramblers on $\G{m}{n}$, that differ only by shifting {\em one vertical} ornament by one step. As in the preceding lemma, the differences $d-r$ and $d'-r'$ are equal, while the ornament sums satisfy $s'=s-1$.

\setcounter{tmp}{\value{Thm}-1}
\begin{customlemma}{{\thetmp}{\sc v}}
 \label{lemma:OrnamentShift-V}
  Consider the scramblers $\O=(H,V)$ and $\O'=(H,V')$ on $\G{m}{n}$, where $V=\{ v_1, \ldots, v_{j-1}, \Hat{v_j\parbox{0.2in}{\footnotesize $-1$}}, v_j, v_{j+1}, \ldots, v_r \}$, and $V' = \{v_1, \ldots, v_j-1, \Hat{\,v_j\,}, v_{j+1}, \ldots, v_r \}$. Then, there is a bijection from $\P{m}{n}^{\O}$ to $\P{m}{n}^{\O'}$ that preserves $\CORNERS$, and decreases $\CINDEX$ by $1$.
\end{customlemma}

\begin{Cor}
 \label{corol:VerticalShiftImpliesMain}
  Formula~\eqref{eqn:Main} is satisfied by $\O$ if and only if it is satisfied by $\O'$.
\end{Cor}
\begin{proof}
  Let $\pi \in \P{m}{n}^{\O}$ and $\pi' \in \P{m}{n}^{\O'}$ be a bijective pair of paths. Lemma~\ref{lemma:OrnamentShift-V} implies that their respective monomials satisfy
  \[
    \l^{\CINDEX(\pi')}\cdot\m^{\CORNERS(\pi')} =
    (\l)^{-1} \cdot \l^{\CINDEX(\pi)}\cdot\m^{\CORNERS(\pi)}.
  \]
  Since shifting a vertical ornament one position to the left implies $s'=s-1$ (while $(d',r')=(d,r)$),
  \[
    \qtbinom{m+n}{m}^{\O'} = (\l)^{-1} \cdot \qtbinom{m+n}{m}^{\O} = 
    \l^{s'} \cdot \sum_{c} \left\{
        \qbinom{m+r'-d'}{c-d'} \cdot \qbinom{n+d'-r'}{c-r'} \cdot \l^{(c-d')(c-r')} \cdot \m^{c}
      \right\}
  \]
\end{proof}

\begin{proof}[Proof of Lemma~\ref{lemma:OrnamentShift-V}]
  The argument follows closely the structure of the proof of Lemma~\ref{lemma:OrnamentPairCancellation}, but with a few differences. By assumption, there are at least two $\R$s, and $H,V \neq \varnothing$, so every word in $\P{m}{n}$ is of the form {\large $\wu\R\,\wa\R\wb\,\wv$}, where
  \begin{itemize}
    \item[--] The $\R$s correspond to the $(v_j-1)$-st and $v_j$-th horizontal steps.
    \item[--] If $\wa$ or $\wb$ are not empty, they contain only $\D$s.
    \item[--] If $\wv$ is not empty, it begins with $\R$.
  \end{itemize}
  Every choice of $\wu,\wv$ determines a class $_{\wu}S_{\wv} \in \P{m}{n}$ of $|\wa\wb|+1$ words. The ornated versions of $\wu,\wv$ are identical for all words in $_{\wu}S_{\wv}^{\O} \cup _{\wu}S_{\wv}^{\O'}$, so we will ignore WLOG any contribution to $\CORNERS$ and $\CINDEX$ originating in $\wu,\wv$. As before, we consider types~\ref{item:type1} through~\ref{item:type4}, noting that this time, type~\ref{item:type3} has the floating letter following $\wu$. Also, in types~\ref{item:type1},~\ref{item:type2}, both $\R$s face a $\D$ or $\oD$, so we can define the same floating orders as before. All in all,
  \begin{itemize}
    \item Each word $\ww \in {_{\wu}S_{\wv}}^{\O}$ of type~\ref{item:type1}~or~\ref{item:type2}, has the same $\CINDEX$ as its counterpart $\ww' \in {_{\wu}S_{\wv}}^{\O'}$.
    \item All words of type~\ref{item:type1} have the same $\CORNERS$ count $c$ in both ${_{\wu}S_{\wv}}^{\O}$ and ${_{\wu}S_{\wv}}^{\O'}$ (here, $c = \# \oD$ if $\wa$ begins with $\oD$, and $c=\#\oD+1$ if $\wa$ begins with $\D$).
    \item All words of type~\ref{item:type2} have the same $\CORNERS$ count $c+1$ in both ${_{\wu}S_{\wv}}^{\O}$ and ${_{\wu}S_{\wv}}^{\O'}$ (because the floating letter creates one more corner by facing an unornated $\D$).
    \item At each step in either list, $\CINDEX$ increases by $1$ (because one $\oD$ shifts one space to the right).
    \vspace{-12pt}
    \item At each step in either list, $\CINDEX$ increases by $1$ (because a $\D\overbrace{\oD\ldots\oD}^k$ block shifts to the left, increasing $\CINDEX$ by $(k+1)-k$).
  \end{itemize}

  String the orders together so that type~\ref{item:type1} precedes type~\ref{item:type3}, which precedes type~\ref{item:type2}. Here we must split into two cases:
  \begin{itemize}
    \item \underline{If $\wa$ begins with $\oD$}, the word $\wu\,\R\Ro\, \wa\wb\, \wv$ of type~\ref{item:type3} in ${_{\wu}S_{\wv}}^{\O}$ ends the floating pattern in type~\ref{item:type1}, so $\CORNERS = \big( \#\oD \big)$, and $\CINDEX$ follows the increase-by-1 rule from type~\ref{item:type1}. Compared to this, the word $\wu\,\Ro\R\, \wa\wb\, \wv$ in ${_{\wu}S_{\wv}}^{\O'}$ has one more corner, so $\CORNERS = \big( \#\oD \big)+1$, just as the following type~\ref{item:type2}. Moreover, this word is the seed to start the $\CINDEX$ increase-by-1 rule of type~\ref{item:type2}.
    \item \underline{If $\wa$ begins with $\D$}, the word $\wu\,\Ro\R\, \wa\wb\, \wv$ of type~\ref{item:type3} in ${_{\wu}S_{\wv}}^{\O'}$ starts the floating pattern in type~\ref{item:type2}, so $\CORNERS = \big( \#\oD \big) +2$, and $\CINDEX$ begins the increase-by-1 rule in type~\ref{item:type2}. Compared to this, the word $\wu\,\R\Ro\, \wa\wb\, \wv$ in ${_{\wu}S_{\wv}}^{\O}$ has one corner less, so $\CORNERS = \big( \#\oD \big)+1$, just as the preceding type~\ref{item:type1}. Moreover, this word terminates the $\CINDEX$ increase-by-1 rule of type~\ref{item:type1}.
  \end{itemize}

  Finally, add the words of type~\ref{item:type4} so that $\wu\,\R \wa\wb\,\Ro\,\wv$ is last in the ${_{\wu}S_{\wv}}^{\O}$ order, and $\wu\,\Ro \wa\wb\,\R\,v$ is first in the ${_{\wu}S_{\wv}}^{\O'}$ order. Then,

  \begin{itemize}
    \item The word $\wu\R\wa\wb\,\Ro\,\wv \in S_{\wv}^{\O}$ has the same number of corners as type~\ref{item:type2}, while $\wa\wb\,\R\,\wv \in S_{\wv}^{\O'}$ has the same number of corners as type~\ref{item:type1}.
    \item The word $\wa\wb\,\Ro\,\wv$ ends the $\CINDEX$ increase-by-1 rule of type~\ref{item:type2}, while $\wa\wb\,\R\,\wv$ starts the increase-by-1 rule of type~\ref{item:type1}.
  \end{itemize}
  
  \begin{table}\begin{center}\begin{tabular}{|c@{\qquad$\longleftrightarrow$\qquad}c|}
  \hline
  $u\,\,\, \R\,\,\, \oD\, \D\, \D\, \oD\, \boxed{\Ro}\, \oD\, \D\,\,\, v$  &
  $u\,\,\, \Ro\,\,\, \oD\, \D\, \D\, \oD\, \oD\, \D\, \boxed{\R}\,\,\, v$ \\
  $u\,\,\, \R\,\,\, \oD\, \D\, \D\, \boxed{\Ro}\, \oD\, \oD\, \D\,\,\, v$  &
  $u\,\,\, \Ro\,\,\, \oD\, \D\, \D\, \oD\,\boxed{\R}\,  \oD\, \D\,\,\, v$ \\
  $u\,\,\, \R\,\,\,\boxed{\Ro}\,  \oD\, \D\, \D\, \oD\, \oD\, \D\,\,\, v$  &
  $u\,\,\, \Ro\,\,\, \oD\, \D\, \D\,\boxed{\R}\,  \oD\, \oD\, \D\,\,\, v$ \\
  $u\,\,\, \R\,\,\, \oD\,\boxed{\Ro}\,  \D\, \D\, \oD\, \oD\, \D\,\,\, v$  &
  $u\,\,\, \Ro\,\,\,\boxed{\R}\,  \oD\, \D\, \D\, \oD\, \oD\, \D\,\,\, v$ \\
  $u\,\,\, \R\,\,\, \oD\, \D\,\boxed{\Ro}\,  \D\, \oD\, \oD\, \D\,\,\, v$  &
  $u\,\,\, \Ro\,\,\, \oD\,\boxed{\R}\, \D\, \D\, \oD\, \oD\, \D\,\,\, v$ \\
  $u\,\,\, \R\,\,\, \oD\, \D\, \D\, \oD\, \oD\,\boxed{\Ro}\,  \D\,\,\, v$  &
  $u\,\,\, \Ro\,\,\, \oD\, \D\,\boxed{\R}\,  \D\, \oD\, \oD\, \D\,\,\, v$ \\
  $u\,\,\, \R\,\,\, \oD\, \D\, \D\, \oD\, \oD\, \D\, \boxed{\Ro}\,\,\, v$  &
  $u\,\,\, \Ro\,\,\, \oD\, \D\, \D\, \oD\, \oD\,\boxed{\R}\,  \D\,\,\, v$ \\
  \hline
  \hline
  \hline
  $u\,\,\, \R\,\,\, \D\, \oD\, \oD\, \D\, \D\,\boxed{\Ro}\,  \oD\,\,\, v$  &
  $u\,\,\, \Ro\,\,\, \D\, \oD\, \oD\, \D\, \D\, \oD\, \boxed{\R}\,\,\, v$ \\
  $u\,\,\, \R\,\,\, \D\, \oD\,\boxed{\Ro}\,  \oD\, \D\, \D\, \oD\,\,\, v$  &
  $u\,\,\, \Ro\,\,\, \D\, \oD\, \oD\, \D\, \D\,\boxed{\R}\,  \oD\,\,\, v$ \\
  $u\,\,\, \R\,\,\, \D\,\boxed{\Ro}\,  \oD\, \oD\, \D\, \D\, \oD\,\,\, v$  &
  $u\,\,\, \Ro\,\,\, \D\, \oD\,\boxed{\R}\,  \oD\, \D\, \D\, \oD\,\,\, v$ \\
  $u\,\,\, \R\,\,\,\boxed{\Ro}\,  \D\, \oD\, \oD\, \D\, \D\, \oD\,\,\, v$  &
  $u\,\,\, \Ro\,\,\, \D\,\boxed{\R}\,  \oD\, \oD\, \D\, \D\, \oD\,\,\, v$ \\
  $u\,\,\, \R\,\,\, \D\, \oD\, \oD\, \boxed{\Ro}\, \D\, \D\, \oD\,\,\, v$  &
  $u\,\,\, \Ro\,\,\,\boxed{\R}\,  \D\, \oD\, \oD\, \D\, \D\, \oD\,\,\, v$ \\
  $u\,\,\, \R\,\,\, \D\, \oD\, \oD\, \D\, \boxed{\Ro}\, \D\, \oD\,\,\, v$  &
  $u\,\,\, \Ro\,\,\, \D\, \oD\, \oD\,\boxed{\R}\,  \D\, \D\, \oD\,\,\, v$ \\
  $u\,\,\, \R\,\,\, \D\, \oD\, \oD\, \D\, \D\, \oD\, \boxed{\Ro}\,\,\, v$  &
  $u\,\,\, \Ro\,\,\, \D\, \oD\, \oD\, \D\,\boxed{\R}\,  \D\, \oD\,\,\, v$ \\
  \hline
\end{tabular}\end{center}
\caption{The floating $\protect\Ro$ trades places with the left $\R$. In the example above, the ``sea'' of $\D$s begins with $\protect\oD$. In the example below, it begins with $\D$.}
 \label{}
\end{table}

  List the words in ${S_{\wv}}^{\O}$ as $\{ \ww_0,\ldots\,\ww_{|ab|} \}$, and those in ${S_{\wv}}^{\O'}$ as $\{ \ww'_0,\ldots\,\ww'_{|ab|} \}$, with indices following the respective orders. Together, the bulleted facts above imply that in both lists, $\CORNERS$ is piecewise constant, jumping up by 1 at the type~\ref{item:type3} word, while $\CINDEX$ increases by 1 at every step except at type~\ref{item:type3}, where the increase is common to both lists. It follows that for every bijective pair $\ww_i \leftrightarrow \ww_i'$, the offset in either statistic is the same regardless of the index $i$.

  To finish the proof, notice that the last two bullets imply $\CORNERS(\ww_i') = \CORNERS(\ww_i)$, and $\CINDEX(\ww_i') = \CINDEX(\ww_i)-1$ for both $i=0$ and $|\wa\wb|$, so the offset is $0$ for $\CORNERS$, and $-1$ for $\CINDEX$ as claimed.
\end{proof}

The exact same argument applies (with obvious modifications) for scramblers that differ in one {\em horizontal} ornament:

\begin{customlemma}{{\thetmp}{\sc h}}
 \label{lemma:OrnamentShift-H}
  Consider the scramblers $\O=(H,V)$ and $\O'=(H',V)$ on $\G{m}{n}$, where $H=\{ h_1, \ldots, h_{j-1}, \Hat{h_j\parbox{0.2in}{\footnotesize $-1$}}, h_j, h_{j+1}, \ldots, h_d \}$, and $H' = \{h_1, \ldots, h_j-1, \Hat{\,h_j\,}, h_{j+1}, \ldots, h_d \}$. Then there is a bijection from $\P{m}{n}^{\O}$ to $\P{m}{n}^{\O'}$ that preserves $\CORNERS$, and decreases $\CINDEX$ by $1$.
\end{customlemma}
\begin{Cor}
 \label{corol:HorizontalShiftImpliesMain}
  Formula~\eqref{eqn:Main} is satisfied by $\O$ if and only if it is satisfied by $\O'$.
\end{Cor}

\subsection{The reduced ornament lemma}
The systematic application of lemmas~\ref{lemma:OrnamentPairCancellation}, and \ref{lemma:OrnamentShift-H},\ref{lemma:OrnamentShift-V} allow us to depopulate scramblers one ornament-pair at a time. Eventually, we are left with a scrambler with only one type of ornaments (either all horizontal, or all vertical), which we can further assume are shifted to their lowest positions. We will prove formula~\eqref{eqn:Main} in this situation:

\begin{Prop}
 \label{prop:ReducedVandermonde}
  Consider the scrambler $\O$ on $G_{m,n}$ with $H = \varnothing$, and $V = \{1,\ldots,r\}$. Then $d=0$, $s = \tbinom{n}{2}$, and the generating function of $(\CINDEX,\CORNERS)$ over all paths in $\P{m}{n}^{\O}$ is
  \begin{equation}
   \label{eqn:ReducedVandermonde}
    \l^{\tbinom{r}{2}} \cdot
    \sum_c \left\{ \qbinom{m+r}{c}\cdot\qbinom{n-r}{c-r} \cdot \l^{c(c-r)}\cdot\m^c \right\}.
  \end{equation}
\end{Prop}
Thus, Formula~\eqref{eqn:Main} holds for this particular scrambler (the case $V=\varnothing$, $H=\{0,\ldots, d-1\}$ is identical); and in light of corollaries \ref{corol:VerticalShiftImpliesMain}, \ref{corol:HorizontalShiftImpliesMain}, this will conclude the proof of Theorem~\ref{thm:Main}.

\begin{Rem}
 \label{rmk:ProveByNumberOfCorners}
  The bijections from lemmas \ref{lemma:OrnamentPairCancellation}, and \ref{lemma:OrnamentShift-H},\ref{lemma:OrnamentShift-V}, are in fact richer than explicitly stated, because they preserve {\em classes of paths with  fixed} $\CORNERS$ count. We will make use of this to establish the correct $q$-binomial count in~\eqref{eqn:ReducedVandermonde} for {\em each} individual class.
\end{Rem}

\begin{Rem}
 \label{rmk:StayWith-d=0}
  A simplifying property in the proof of Proposition~\ref{prop:ReducedVandermonde} is the trivial fact that a path with $c \geq d+r$ must have at least $\ell := c-(d+r) \geq 0$ {\em true} corners. We are assuming $d=0$, so that {\em every} path in $\P{m}{n}^{\O}$ satisfies $\ell \geq 0$, but it is easy to extend the argument below to scramblers of mixed type, as long as a given class satisfies $c \geq d+r$. Simply add $d$ ersatz $\R$s, for a total of $n+d-r$ $\R$s. When distributing the $\ell$ plain $\D$s to reconstruct the paths, we get $\ell$ $\R\D$-strings as in the proof, but now $d$ of these can be replaced by $\oD$s, eliminating the fake $\R$s, and realizing the correct number of $\oD$s. It is possible to prove~\eqref{eqn:Main}, even without assuming $c \geq d+r$, but the proof becomes more complicated than necessary.
\end{Rem}

\begin{proof}[Proof of Proposition~\ref{prop:ReducedVandermonde}]
  As suggested in remark~\ref{rmk:ProveByNumberOfCorners}, we will show that the generating function of $\CINDEX$, restricted to paths with exactly $c$ corners is
  \[\qbinom{m+r}{c}\cdot\qbinom{n-r}{c-r} \cdot \l^{c(c-r) + \tbinom{r}{2}}.\]
Then, Formula~\eqref{eqn:ReducedVandermonde} follows in a straightforward manner, inserting the factor $\m^c$ to keep track of the $\CORNERS$ count, and adding over all values of $c$. \\

Let $\ell = c-r \geq 0$ as in remark~\ref{rmk:StayWith-d=0}. The value $\ell$ can be interpreted as the number of unornated corners; here, this value remains unchanged for every path under consideration, because $d=0$. In terms of words, every path has $r$ $'\Ro'$ strings, and $\ell$ $'\R\D'$ strings and, because the ornaments are shifted to the left, every $\Ro$ appears {\em before} every $\R\D$; in other words, the $r$ ornated corners (true and/or virtual) appear {\em before} the $\ell$ unornated corners (all of which are true corners).

In fact, the $\Ro$s are located to the left of all plain $\R$s, so a word in $\P{m}{n}^{\O}$ consists of the string
\begin{equation}
 \label{eqn:Rs}
  \underbrace{\Ro \ldots \Ro}_{r} \underbrace{\R \ldots \R}_{n-r},
\end{equation}
with $m$ $\D$s interleaved. We propose to reconstruct the words in $\P{m}{n}^{\O}$ in two stages:
\renewcommand{\labelenumi}{\alph{enumi})}
\begin{enumerate}
  \item Start with the string~\eqref{eqn:Rs}, and insert $\ell$ $\D$s among the $n-r$ plain $\R$s, so that no two $\D$s come together; this creates $\ell$ true corners, for the correct total of $c$. The number of such combinations is $\binom{n-r}{\ell} = \binom{n-r}{c-r}$.
  \item Insert the remaining $m-\ell$ $\D$s so that no spurious corners are created; this means distributing them among the $r$ $\Ro$s and the $\ell$ $\R\D$ strings (in this stage, multiple $\D$s can lie in succession). The number of such combinations is $\binom{(m-\ell) + (r+\ell)}{r+\ell} = \binom{m+r}{c}$.
\end{enumerate}
There is a trivial bijection between the words created in step a) and the words of $\P{n-r-\ell}{\ell}$: simply ignore all $\Ro$s, and substitute every $\R\D$ string with a $\D$. A swap $\D\,\R\D \leftrightarrow \R\D\,\D$ corresponds to a swap $\D\,\R \leftrightarrow \R\,\D$ in $\P{n-r-\ell}{\ell}$ (cf. \S\ref{sect:Definitions}), and moreover, they elicit the {\em same} change in $\CINDEX$. It follows that the generating function of $\CINDEX$ among the words from step a) is
\[\qbinom{n-r}{\ell} \cdot \l^x\]
for some $x$.

For every word from step a), consider now the corresponding set of words created in step b). This time we have a bijection to the words in $\P{m-\ell}{r+\ell}^{\O}$ by ignoring the inserted $\D$s, and those plain $\R$s that have not formed a corner (notice we {\em do} pay attention to ornaments now). This time, the swaps $\Ro\,\D \leftrightarrow \D\,\Ro$, and $\R\D\,\D \leftrightarrow \D\,\R\D$ elicit the correct $\CINDEX$ changes that imply the generating function of $\CINDEX$ (restricted to the newly inserted $\D$s) is
\[\qbinom{(m-\ell)+(r+\ell)}{r+\ell} \cdot \l^y\]
for some $y$.

Since the choices made in stages a) and b) can be made independently, the generating function for $\CINDEX$ on all words of $\P{m}{n}^{\O}$ is the product of the two binomial terms we have found:
\[\qbinom{n-r}{c-r} \cdot \qbinom{m+r}{c} \cdot \l^z,\]
and it only remains to evaluate the constant shift $z$. However, $z$ is the smallest $\CINDEX$ value, and this is attained by the word
\[
  \underbrace{\Ro \ldots \Ro}_{r} \underbrace{\R\D \ldots \R\D}_{\ell}
  \underbrace{\D \ldots \D}_{m-d-\ell} \underbrace{\R \ldots \R}_{n-r-\ell}.
\]
from where it is clear that $z = (c-r)^2 + (c-r)r + \binom{r+1}{2} = (c-r)c + \binom{r+1}{2}$, thus concluding the proof.
\end{proof}

\subsection{$\boldsymbol\AREA$ revisited, and further ideas}
\label{sect:Final}
We would like to conclude with two observations. First, Theorem~\ref{thm:Main} implies all our early claims about equidistribution of statistics, but there is a more explicit connection between ornated corners and areas:

The distribution of $\CORNERS$ on unornated paths in $\P{m}{n}$ is $\sum_c \binom{m}{c}\binom{n}{c} \cdot \m^c$. On the other hand, if $\O = (\{0,\ldots,m-1\},\varnothing)$ is the scrambler with all horizontal, and no vertical ornaments, then every path $p \in \P{m}{n}^{\O}$ has $m$ corners, and satisfies $\CINDEX(p) = \AREA(p) + \tfrac{(m-2)(m-1)}{2}$. From these two observations it follows easily that the generating function of $(\AREA,\CORNERS)$ on $\P{m}{n}^{\O}$ is
\[\sum_c \qbinom{m}{c}\qbinom{n}{c} \cdot \m^c.\]

Second, we are indebted to Niklas Eriksen for suggesting an alternative interpretation of the effect of ornaments, so they interact with corners of $\D\R$ type, as well as with ``normal'' corners of $\R\D$ type. The results are exactly the same, and the presentation seems more symmetrical, even if the ornaments still have a preferred direction of facing paths (to the left for horizontal, and downward for vertical). We expect that variations on the theme of corners and ornaments may still offer interesting results.

\bibliographystyle{plain}
\bibliography{ref}

\end{document}